\documentclass{article}
\usepackage{amsmath,amssymb,latexsym,amsthm,color}

\newcommand{\KK}{\mathbb{K}}
\newcommand{\QQ}{\mathbb{Q}}
\newcommand{\RR}{\mathbb{R}}
\newcommand{\NN}{\mathbb{N}}

\newcommand{\coker}{\mathrm{coker}\kern.5pt}

\theoremstyle{plain}
\newtheorem{theorem}{Theorem}[section]
\newtheorem{corollary}[theorem]{Corollary}
\newtheorem{proposition}[theorem]{Proposition}
\newtheorem{lemma}[theorem]{Lemma}

\theoremstyle{definition}

\newtheorem{example}[theorem]{Example}

\newtheorem{remark}[theorem]{Remark}

\def \e {\epsilon}

\begin{document}

\title{Double series over a non-Archimedean field}

\author{L. Corgnier, C. Massaza, P. Valabrega}
 
\date{}

\maketitle

\section{Abstract}

 In the present paper we  investigate the convergence of a double series
over a complete non-Archimedean field  and prove that, while the  proofs  are somewhat different, the Archimedean results hold true.

\textbf{Keywords:}  non-Archimedean fields, power series, double series.
\textbf{MSC\,2010:} 12J15.

\section{Introduction}

Double series, in particular double power series, over the real (or complex) field, i.e. under archimedean hypotheses, have been thoroughly investigated. The main results on the reordering of such a series can be found in \cite{Goursat}, \textsection 169, where it is proved that a double series over the field of real numbers, under suitable conditions of absolute convergence, can be arranged in (infinitely) many ways, without changing its sum. The proof is strongly based on the least upper bound property, which fails to be true over a non-Archimedean (Cauchy complete) ordered field.

The aim of this paper is the investigation of such series over a complete non-Archimedean field. Indeed we can show that, with an approach that does not involve any least upper bound property, the classical results hold true.

Since the rearranging of the terms of a double series reminds us of the famous Fubini Theorem (concerning double integrals), we use the term Fubini-type Theorem.

We investigate in particular the substitution of a power series into another power series, an operation that gives rise to a double series, so that the stated results can be applied to this case. We also show by an example  that, if our conditions of absolute convergence  fail to be true, then the reordering is not allowed. 

When we substitute a power series $S(X)$ into another power series $T(X)$, actually two different object have to be considered: the power series $T(S(X))$, whose coefficients are formally obtained by algebraic manipulation, and the composite function $T\circ S$. It is not obvious that these two objects lead to the same domain of definition. In fact we can show by suitable examples that the power series $T(S(X))$ may have a domain which is either strictly larger or striclty smaller than the domain of  $T \circ S$.

Our results can also be applied to produce a concrete example of a power series having a well-identified non-algebraic zero.

\section{Preliminaries and notation}

$\KK$ is a complete non-Archimedean field with a countable basis for the neighbourhoods of $0$ (see \cite {Kelley} ).  In some examples we will consider the complete non-Archimedean ordered field $\KK = \widehat{\QQ(\e)}$, where $\e$ is infinitesimal (see ...). 

 A series $\sum_{n = 0}^{\infty}c_n$ converges to $c \in \KK$ if the sequence of the partial sums converges to $c$. The series is absolutely convergent if $\sum_{n = 0}^{\infty}|c_n|$ is a converging series. It is known (\cite {Massaza},...) that $\sum_{n = 0}^{\infty} c_n$ is convergent if and only if $\lim c_n = 0$, and therefore a series is convergent if and only if it is absolutely convergent (which is false over  $\RR$).
 
 Obviously the sum term by term of two converging series is also a converging series, whose sum is the sum of the two series. Moreover, if we arbitrarily associate the terms of a converging series, we obtain a new series that converges to the same sum (we are considering a subsequence of the sequence of partial sums).
 
 From now on $\sum c_n$ means $\sum_{n = 0}^{\infty} c_n$.
 
 \begin{lemma} \label{Splitting}

The sum of every convergent series  can be obtained as the  difference of two converging  series, each one with non-negative terms.
\end{lemma}

\begin{proof}
 Assume that $\sum c_n$ is a converging series. Let us set: $$c_n^+ = \frac{|c_n|+c_n}{2}, c_n^- = \frac{|c_n|-c_n}{2}$$  Then $\sum c_n^+, \sum c_n^-$ are converging series with non-negative terms and, using the above property on sum of series, $$\sum c_n = \sum c_n^+ - \sum c_n^-$$
 
 \end{proof}
 
 A {\it{reordering}} of the series $C = \sum c_n$ is by definition a series $\sum d_n$, where $d_n = c_{f(n)}, \forall n \in \NN$,  $f: \NN \to \NN$  being a one-to-one function .
 
 \begin{proposition} \label {Reordering}

Let $C = \sum c_n$ be a converging series. Then every series obtained from $C$ by a reordering of its terms is still a converging series having the same sum. 
\end{proposition}
\begin {proof}
We first assume that $c_n \ge 0, \forall n$. 

Let $f: \NN \to \NN$ be the one-to-one function that defines the reordering, so that we are considering  the new series $\sum d_n$ where $d_n =  c_{f(n)}$. It is easy to see that each partial sum of this series  $D_n = \sum_{i = 0}^n d_n$ is less than or equal to a suitable partial sum of the original series $C_m$, and conversely, which implies the result.

In general, it is enough to use the equality $\sum c_n = \sum c_n^+ - \sum c_n^-$ (see \ref{Splitting}) and the fact just proved.
\end{proof}
 \begin{remark} The statements of Lemma \ref{Splitting} and of Proposition \ref{Reordering} in the archimedean case hold true only for an absolutely converging series (see \cite{Goursat}, ...). 
 \end{remark}

\begin{remark} 
The proof of the above proposition shows that performing on a series any reordering  obtained  through a suitable bijection of 
$\NN$ is allowed. It is however not allowed to transform a series into a series of series, since the ordering in a series of series has an ordinal number $\omega^2$, which is not the same as the ordinal number $\omega$ of $\NN$. The transformation of a series into a series of series and conversely is studied in the following chapter.
\end{remark}

\begin{remark} \label{Comparison}
 If $\sum a_n$ is convergent and $|b_n| \le |a_n|, \forall n$, then also $\sum b_n$ is convergent. It is in fact enough to observe that $\lim |b_n| \le \lim |a_n|  = 0$.
\end{remark}

\section{Double series and  Fubini-type theorems}

Let $a_{ij}$ be any set of elements in $\KK$, where $(i,j) \in \NN^2$. Choose a bijection $f: \NN \to \NN^2$ and set: $f^{-1} = (\phi,\psi)$ = inverse function $\NN^2 \to \NN \ (i = \phi(k), j = \psi(k))$. Then it makes sense to consider the series $\sum_k a_{\phi(k),\psi(k)}$. If it is a converging (hence an absolutely converging) series, we set by definition: 
$$\sum a_{ij} = \sum a_{\phi(k),\psi(k)}$$  Since any reordering of the terms of a converging series is allowed (and does not change the sum), every bijection $f$, together with its inverse function, gives rise to the same sum of the series.

\begin{remark}
Since a double series is transformed by the definition into an ordinary simple series, all the properties of the preceding section apply to double series too.
\end{remark}

\begin{remark}

In the literature (\cite {Goursat}, \textsection 169) sometimes the definition of  convergence and sum of a double series is given in a different way (in the following we call it   "Goursat"). In order to verify that Goursat's definition   and ours are equivalent, we briefly recall \cite{Goursat}, \textsection 169 (where the series is  absolutely converging over $\RR$).

Consider the set $ \NN^2$, and a countable sequence $I_0, I_1, I_2, .... $ of finite subsets, satisfying the following two conditions:

1-  $I_n \subset I_{n+1}, \forall n$

2-  $\cup _nI_n = I$.

Now set: $\forall k = 0, 1 , \cdot \cdot \cdot, S_k = \sum a_{ij}, (i,j) \in I_k$.

 If this sequence is absolutely convergent, we say that  the double series is convergent to $\lim S_k$.
The definition is correct, since it is proved in (\cite {Goursat}, \textsection 169) that it  is independent on the chosen sequence of subsets, under the hypothesis, always satisfied in the non-Archimedean case, that the convergence is absolute.

Now consider, according to our definition of  a double series, a one-to-one map $f: \NN \to \NN^2$ and its inverse   $f^{-1}: \NN^2 \to \NN$ $((i,j) = f^{-1}(k)  = \phi(k), j = \psi(k))$. Define a Goursat sequence of subsets according to the following iterative rule:

$I_0 = \{(\phi(0),\psi(0))\}, I_{k+1} = I_k \cup \{(\phi(k+1),\psi(k+1))\}$.

With these choices it is readily seen that the sequence  $S_0,S_1,S_2,....$ coincides with the sequence of the partial sums of the series $\sum a_{\phi(k),\psi(k)}$, and therefore they are both convergent (to the same limit) or not convergent.

\end{remark}

We are now interested in the following two rearrangings  of the terms of the series $\sum a_{ij}$: $\sum_i(\sum_j a_{ij})$ and $\sum_j(\sum_i a_{ij})$. In fact what we are now going to perform is not a reordering, as introduced above: it does not depend on a one-to-one map $\NN \to \NN$,  since it involves $\NN^2$.  

If the terms of the double series form an infinite matrix:

$a_{00}  \ a_{01} \cdot \cdot \cdot$ 

$a_{10} \ a_{11} \cdot \cdot \cdot $

$\cdot \cdot \cdot$ 

$a_{i0} \ a_{i1} \cdot \cdot \cdot  \ a_{ij}  \cdot \cdot \cdot$

$\cdot \cdot \cdot ,$

the two rearrangings correspond to a sum by rows or by columns (both being series).

\begin{theorem}{\label{Fubini}} (Fubini-type theorem) Assume that $\sum a_{ij}$ is a converging series. Then, $\forall i, \sum_j a_{ij}$ and $\sum_i(\sum_j a_{ij})$ are convergent. Moreover $\sum_i(\sum_j a_{ij}) = \sum_{ij} a_{ij}$.
\end{theorem}

\begin{proof} Let us start with the case $a_{ij} \ge 0, \forall i,j$ and set: $A = \sum a_{ij}$. Obviously it holds:  $ \sum_{j = 0}^m a_{ij} \le A$ for every $i$.

On an Archimedean field this is enough to ensure the convergence of
 $ \sum a_{ij} $ for every $i$, because the partial sums are increasing and bounded above. In our non-Archimedean hypothesis, we choose any integer $i_0$ and set:

$b_{ij} = a_{ij}$ if $i = i_0$, $0$ otherwise.

Since $0 \le b_{ij} \le a_{ij}, \forall i,j, \sum b_{ij}$ is convergent, But $\sum b_{ij} = \sum_j a_{i_0j}$, so that $\sum_j a_{ij}$ is convergent for every $i$. From the inequality $\sum_{i = 0}^n(\sum_{j = 0}^m a_{ij}) \le A$ we obtain, when $m$ tends to $\infty$, $\sum_{i = 0}^n(\sum_j a_{ij}) \le A$. 

Neither convergence when $n$ tends to infinity nor equality can now be deduced, unlike it can be done in the Archimedean case. 

Therefore we proceed as follows. Given any $\e > 0$, it holds, if  $n$ is large enough: $\sum_{i = 0}^n(\sum_j a_{ij}) \ge \sum_{i = 0}^n(\sum_{j = 0}^n a_{ij}) \ge A-\e$, since, if $n$ is large enough, $\sum_{i = 0}^n(\sum_{j = 0}^n a_{ij})$ contains all the terms of any given partial sum of the double series $\sum_{ij} a_{ij} = A.$ The two inequalities $A-\e \le \sum_{i = 0}^n(\sum_j a_{ij}) \le A$ give the required result: $\sum_i(\sum_j a_{ij}) = A$.

If $a_{ij}$ is not positive for every pair $(i,j)$, it is enough to use $a^+_{ij}$ and $a^-_{ij}$ as introduced in Lemma \ref{Splitting} and perform some calculation using the already obtained result for series with non-negative terms.

\end{proof}

\begin{remark} Of course the Fubini-type theorem can be stated as follows: $\sum_i(\sum_j a_{ij}) = \sum_j(\sum_i a_{ij}) = \sum_{i,j} a_{ij}$ (the sum can be computed either by rows or by columns).
\end{remark}

A converse Fubini-type theorem can be stated as follows.

\begin{theorem} {\label{Fubini converse}} With the notation above, assume that both
$\sum_j a_{ij}, \forall j,$ and $\sum_i( \sum_j( |a_{ij}|)$ converge; then $\sum_{i,j} a_{ij}$ also converges, as well as $\sum_j( \sum_i a_{ij})$. Moreover it holds: $\sum_{i,j} a_{ij} = \sum_i( \sum_j a_{ij}) = \sum_j( \sum_i a_{ij})$
\end{theorem}

\begin{proof} As usual we start with the supplementary condition $a_{ij}\ge 0, \forall i,j$. Let $R$ be a partial sum of the series $\sum_{i,j} a_{ij}$ (the terms being arranged in some order). Assume that, in $R, i \le n, j \le m$. Then $R \le \sum _{i=0}^n(\sum_{j=0}^m a_{ij}) \le \sum _i(\sum_j a_{ij})$. In the Archimedean case, this would obviously imply the expected convergence of the sequence of the partial sums, and so of the double series. 
In our case we set: $b_i = \sum_j a_{ij}, A = \sum_i b_i$ (they are converging series by hypothesis). 
If we choose any $\e > 0$, there are $n$ such that $\sum_{i = 0}^n b_i > A - \e$ and integers $m_0,m_1,\cdot \cdot \cdot,m_n$ such that $\sum_{j=0}^{m_i} a_{ij} > b_i-\e$. If $\omega \in \KK$ is any element larger than any integer, we obtain: $\sum_{i=0}^n(\sum_{j=0}^{m_i} a_{ij}) > \sum_{i=0}^n(b_i-\e) = \sum_{i=0}^n b_i-(n+1)\e > A-\e-(n+1)\e = A-(n+2)\e > A-\omega \e$. Therefore there are partial sums of the double series less than $A$ but as close as we want to $A$, which means that the increasing sequence of the partial sums tends to $A$.

The general case ($a_{ij}$ not necessarily positive) follows easily. In fact $\sum_i(\sum_j|a_{ij}|)$ is convergent by hypothesis, hence the double series of the absolute values is also convergent, as just proved, and so also $\sum_{i,j}a_{ij}$.  Now the required equality follows from   Theorem \ref{Fubini}.
\end{proof}

\begin{remark} Observe that we require the convergence of the series $\sum_i(\sum_j |a_{ij}|)$, and our proof does not work under the weaker condition of convergence of both $\sum_j|a_{ij}|$ and $\sum_i|\sum_j a_{ij}|$, or equivalently of both $\sum_ja_{ij}$ and $\sum_i\sum_j a_{ij}$. This is shown by the following example.

\end{remark}

\begin{example} Let $\sum_i k_i$ be a series converging to $1$ and such that $k_i > 0, \forall i$.

Now set: $a_{i0} = 1-k_0, a_{ij} = -k_j$ if $j \ne 0$. Therefore $a_{ij}$ does not depend upon $i$. With the notation above, we have: $b_i = a_{i0}+a_{i1}+\cdot \cdot \cdot = 0$, so that $\sum b_i = 0$. However the series $\sum_i a_{i0} = a_{00}+a_{10}+\cdot \cdot \cdot $ is not a converging series. 
This is not in contradiction with the preceding theorem.
In fact observe that $|a_{i0}|+|a_{i1}|+ \cdot \cdot \cdot $ converges to $2(1-k_0)$ $(\forall i)$, and  $\sum_i 2(1-k_0)$ does not converge to any value.
\end{example}

As a corollary of \ref{Fubini converse} we have an application to the product of series.

\begin{corollary} {\label {product}} Assume that $\sum b_i$ converges to $B$ and $\sum c_i$ converges to $C$. Then the double series $\sum_{i,j} b_i c_j$ converges to $BC$.
\end{corollary}

\begin{proof} Let us set: $a_{ij} = b_ic_j$. We want to apply Theorem \ref{Fubini converse} to the double series $\sum_{i,j} a_{ij}$. So we consider $\sum_i(\sum_j |a_{ij}|) = \sum_i(\sum_j|b_i||c_j|)$, where the internal series converge $\forall i$, by hypothesis, to $|b_i| \sum_j|c_j|$, while the external series $\sum_i(|b_i| \sum_j|c_j|) $ converges, by hypothesis, to $(\sum_j |c_j|)(\sum_i|b_i|)$. We conclude that $\sum_{i,j} (b_ic_j) = \sum_i(\sum_j(b_ic_j) = (\sum_ib_i)(\sum_j c_j)$.
\end{proof}

The above corollary on product of series has been used, with an ad-hoc direct proof working only in the  case under investigation, in \cite {CMV1}.

\begin {remark}

Observe that, on the field of real numbers, the property of Corollary \ref{product} can be proved under the supplementary assumption that at least one of the two series is absolutely convergent.

\end  {remark}

\section{A generalisation of the Fubini-type theorem}
Consider the converging double series $\sum_{i,j} a_{ij}$ , where  $(i,j) \in \NN^2$. Let $J$ be a subset of $\NN^2$, finite or countable. The sum restricted to $J$ is defined by introducing

$b_{ij}=a_{ij}$  if$ (i,j) \in J$,    $b_{ij}=0$  if$ (i,j) \notin J$

andsetting by definition:
$$\sum_J a_{ij}=\sum_{i,j} b_{ij}$$
(the convergence is ensured by the fact that $ |b_{ij}| \le |a_{ij}|)$.

\begin{theorem} {\label {Fubini gen}} Assume that  $\NN^2$ is partitioned into a finite or countable number of disjoints subsets $J_0,J_1,J_2....$  , each finite or countable. 
Then $$\sum_{i,j} a_{ij}=\sum_{r = 0}^{\infty} (\sum_{(h,k) \in J_r} a_{hk})$$
\end{theorem}

\begin{proof} {\label {Fubini gen}}
First of all we choose an order inside each of the subsets $J_i$ (here we consider the case  of $J_i$ countable, in the case of $J_i$ finite the proof is simpler). 
Then define $c_{ij}$  to be the $j^{th}$ element of $J_i$  in the chosen ordering.
It is easy to see that the matrix $c_{ij}$ is a reordering of the matrix $a_{ij}$. Therefore  $\sum_{i,j} a_{ij}=\sum_{i,j} c_{ij}$ and an application of theorem \ref{Fubini} provides the result.
\end{proof}

\section{Power series}

\subsection{Product of power series}

The results of section 4 can be applied to the product of two power series, obtaining the following 

\begin{corollary}\label {power product} If the two power series $\sum_ia_iX^i,  \sum_jb_jX^j$ converge at the same $x \in \KK$, then, $(\sum_ia_ix^i)( \sum_jb_jx^j) = \sum_k(\sum_{i=0}^ka_ib_{k-i})x^k$, (where the latter series is convergent at $x$).
 
\end{corollary}
\begin{proof}
By Corollary \ref{product}, the double series $\sum_{ij}a_ib_jx^ix^j$ is convergent to the value $(\sum_ia_ix^i)( \sum_jb_jx^j)$. Moreover, since it is allowed to reorder the terms of a double series by increasing values of $i+j$, we obtain the following formula: $\sum_{i,j}a_ib_jx^ix^j = \sum_k(\sum_{i+j = k}a_ib_j)x^k$.

\end{proof}

As for the $i^{th}$ power of a power series, we have the following 

\begin{corollary} \label {power power}Let $S(X) = \sum_ja_jX^j$ be a power series converging at $x$. Then $(\sum_ja_jx^j)^i = \sum_j c_{ij} x^j$, where the $c_{ij}$'s are given by the recursive formulae:  

$c_{00} = 1, c_{0j} = 0 \ \forall j \ne 0 $

$ c_{i+1 \ j} = c_{i0}a_j+c_{i1}a_{j-1}+\cdot \cdot \cdot+c_{ij}a_0 $

\end{corollary}
\begin{proof} It is enough to apply $i$ times Corollary \ref{power product}.
\end{proof}

\subsection{Substitution into a power series}

Let  $S(X) = \sum_j a_jX^j, T(Y) = \sum b_iY^i$ be two power series. If we replace $Y$ by $S(X)$ into the series $T(Y)$,  we obtain a new object $T(S(X))$ that can be transformed into a power series (with variable $X$) if we perform all the usual operations on the coefficients as if it were allowed to gather all coefficients of $X^n, \forall n = 0, 1, \cdot \cdot \cdot$, disregrading the (minor) problem that $S(X), T(Y)$ are power series and not  polynomials. More precisely, we consider $T(S(X))$ as a new element of $\KK[[X]]$, say $T(S(X)) = \sum _nd_nX^n$, where 

$d_j=\sum_ib_ic_{ij} \ (*)$. 

Observe that we are dealing with series whose convergence has to be checked and that the infinite matrix $c_{ij}$ is given by the recursive formulae of Corollary \ref{power power}.

Such a procedure, which is quite natural, makes sense if the series $d_n, \forall n$ attain a value in the complete field $\KK$, i.e. if they are converging series. In this event the replacement gives rise to a new power series, which we can call $T(S(X))$.

In order to investigate such a replacement, mainly  focusing our attention on convergence, we call the $d_n$'s satisfying (*)  {{$\mathbf{expected \ coefficients}$}. The following theorem throws light into the replacement.

\begin{theorem} \label{substitution} 
Let $S(X) = \sum_ja_jX^j, T(X) = \sum_ib_iX^i$ be two power series such that:

(i) $S(X)$ has a sum $S(x) = k$ at some $x \in \KK$, and consequently also $ \sum_j|a_j||x^j|$ has a sum $\bar k$;

(ii) $T(X)$ converges to $T(\bar k)$ at $\bar k = \sum_j |a_j||x^j|$.

Then 

(i) $d_j, \forall j,$ is a convergent series

(ii) $T(k)=  \sum_jd_jx^j$, where the  $d_j$'s are  the {$\mathbf{expected \ coefficients}$ and are series converging in $\KK$}.

\end{theorem}

\begin{proof} 
Case  A.

Assume that $a_i \ge 0, \forall i, x > 0$ and put: $k = S(x)$. 
Then $T(k) = T(S(x)) = \sum_ib_i(\sum_j a_jx^j)^i$. 
By applying Corollary \ref{power power},  the power appearing in the last formula can be expressed as a convergent power series:
$$(\sum_j a_jx^j)^i=\sum_j c_{ij} x^j$$
where the  (obviously non-negative) coefficients $c_{ij}$ are given by Corollary \ref{power power}.

Therefore we have: $T(k) = \sum_ib_i(\sum_jc_{ij}x^j)$ and can apply \ref{Fubini converse} to the double series $\sum_{ij}b_i c_{ij}x^j$. In fact the two conditions of \ref{Fubini converse} are fulfilled:

(i) $\sum_j|b_ic_{ij}x^j| = |b_i| (\sum_j c_{ij}x^j)$ converges to $|b_i|(\sum_ja_jx^j)^i$;

(ii) $\sum_i|b_i|(\sum_ja_jx^j)^i)$ is convergent by hypothesis.

Therefore we obtain the expected equality: $$T(k) = \sum_ib_i(\sum_ja_{ij}x^j) = \sum_j(\sum_ib_ic_{ij})x^j=\sum_j d_j x^j$$
where $d_j$ is given by the convergent series $d_j=\sum_ib_ic_{ij}$

Case B. 

Assume that $x > 0$ as for Case A, but don't put any condition on the sign of the $a_j's$. 

As before, put $k=S(x)$    and $\bar k=\bar S(x)=\sum_j |a_j|x^j$.
Obviously $\bar k \ge\ |k|$.
By hypothesis $T(\bar k)$ is convergent, therefore the results of Case A can be applied to expand $T(\bar k)$ in power series of $x$, since all the coefficients are non negative in the series $\bar S(x)$. 
Starting with the coefficients $|a_j|$ instead of $a_j$, the quantities $c_{ij}$ are replaced by new quantities $\bar c_{ij}$ , and an observation of the recurrence relations of Corollary \ref{power power} says that $|c_{ij}| \le \bar c_{ij}$.

By the same calculation of Case A we have that $\sum |b_i|(\sum_j \ |a_{j}|x^j)^i=\sum |b_i|(\sum_j \bar c_{ij}x^j)$ is convergent.
Therefore  in the repeated series $\sum_i |b_i|(\sum_j  |c_{ij}|x^j)$ all the terms are non negative and not greater then the corresponding terms of a converging series. Then  $\sum_i |b_i|(\sum_j  |c_{ij}|x^j)$ is convergent. Then also the double series  $\sum_{ij} |b_i|  |c_{ij}|x^j$ is convergent, by Theorem \ref {Fubini converse}.

In conclusion, this implies the convergence of $\sum_{ij} b_i  c_{ij}x^j$, and an application of Theorem \ref {Fubini} allows to obtain the result.

Case C

Finally, assume that $x <0$, and define $y=-x$.

We have $S(T(x))=S(T(-y))=\sum_i b_i (\sum_j (-1)^j a_j y^j)$. By hypothesis $S(X)$ is convergent at $\bar k=\sum_j |a_j| |x^j|=\sum_j |a_j|y^j=\sum_j |(-1)^j a_j|y^j$.
Therefore Case B can be applied, and the result follows from an  immediate calculation.

\end{proof}

\begin{remark} Observe that we require the convergence of the series $\sum_i b_i(\sum_j |a_{j}||x^j|)^i$, and the claim  does not hold true under the weaker condition of convergence of both $\sum_j |a_{j}||x^j|$ and $\sum_i b_i|(\sum_j a_{j}x^j)^i|$, as the following example shows.

\end{remark}

\begin {example} \label{non-substitution}
Let  $\varepsilon $ be a positive topologically nilpotent infinitesimal, i.e.  \[\mathop {\lim }\limits_{n \to \infty } {\varepsilon ^n} = 0\]  (we  suppose that it  exists, otherwise the question on convergence of series are trivial, since any power series is convergent everywhere or only at $0$, see \cite{CMV1}).

Consider the series  
\[S(X) = {a_0} + {a_1}X + {a_2}{X^2} + ....\]

where
\[{a_0} = \frac{1}{{1 - \varepsilon }} - 1 = \frac{\varepsilon }{{1 - \varepsilon }}\,\,\,\,\,\,\,\,{a_1} = {a_2} = {a_3} = ..... =  - 1\]    

It is immediate that $S(X)$ is convergent at $x=\varepsilon$, with sum $0$ .

Consider the series

\[T(X) = 1 + \frac{X}{\varepsilon } + {(\frac{X}{\varepsilon })^2} + ....\]

It is convergent, for instance, at $X=\varepsilon^2$. Moreover we have $T(S(\varepsilon )) = 1$.

Therefore one could expect that the formal substitution of $S(X)$  into $T(X)$   would produce a power series in $X$, convergent at $\varepsilon $, with sum $1$. 

This is not the case; in fact the formal substitution provides:
\[T(S(X)) = 1 + \frac{{{a_0} + {a_1}X + {a_2}{X^2} + .....}}{\varepsilon } + {(\frac{{{a_0} + {a_1}X + {a_2}{X^2} + .....}}{\varepsilon })^2} + ....\]

A reordering according to increasing powers of $X$ is impossible; in fact the coefficient of the power $0$ would be
\[1 + \frac{{{a_0}}}{\varepsilon } + {(\frac{{{a_0}}}{\varepsilon })^2} + .... = 1 + \frac{1}{{1 - \varepsilon }} + \frac{1}{{{{(1 - \varepsilon )}^2}}} + .....\]
and this series is not convergent, since $\frac{1}{{1 - \varepsilon }}$ is different from $1$ for an infinitesimal quantity, and consequently its $n-$th power does not tend to $0$. 

This is not in contracdition with the last theorem, because the criterion on the absolute convergence is not satisfied by the example. In fact, consider the series $\bar S(X)$ obtained from $S(X)$ by taking  the absolute values of the coefficients and calculate it at $x=\varepsilon$, obtaining:
 \[\overline S (x) = |{a_0}| + |{a_1}|x + |{a_2}|{x^2} + .... = \frac{\varepsilon }{{1 - \varepsilon }} + \varepsilon + {\varepsilon^2} + ....=2\frac{\varepsilon }{{1 - \varepsilon }}\]

Finally note that the series $T(X)$ is not convergent at $2\frac{\varepsilon }{{1 - \varepsilon }}$.
In fact 
\[ 1 + \frac{2}{{1 - \varepsilon }} + {(\frac{2}{{1 - \varepsilon }})^2} + .....\]   is not convergent.

\

\end {example}

The following result follows from Theorem \ref{substitution}.  

\begin {theorem}\label{substitution2}

If the convergence domains of both $S(X)$ and $T(X)$ are different from the set $\{0\}$, and if  $T(X)$ is convergent at $a_0$, then in a suitable neighbourhood of $0$, $T(S(x))$ is given by the series
$T(k)=  \sum_jd_jx^j$, where the  $d_j$'s are  the $\mathbf{expected \ coefficients}$ and are series converging in $\KK$.

\end {theorem}

\begin {proof}

Case 1: $a_0 \ne 0$. We know that $\sum_i b_i a_0^i $ is convergent by hypothesis, and therefore also $\sum_i |b_i| |a_0|^i$. Moreover $\sum_j |a_j| |x|^j$ is convergent for some $x\ne0$, and therefore in a neighborough of $0$.  It is continuous and takes the value $|a_0|$ at $x=0$, then for $|x|<k$ (where $k$ is a suitable value), we obtain $\sum_j |a_j| |x|^j<2|a_0|$ .

Since the series $\sum_i b_i X^i $ is convergent at $|a_0|$, it is convergent also at  $2|a_0|$ (see \cite{CMV1}), then $\sum_i b_i (\sum_j |a_j| |x|^j  )^i$ is convergent. The conclusion follows from Theorem  \ref{substitution}.

Case 2: $a_o = 0$. Let $a \ne 0$ be any point where $T(X)$ is convergent. Since $S(0) = 0$ and $S(X)$ is a continuous function, in a suitable neighbourhood of $0$ it holds: $|S(x)| < |a|$. Then we argue as above.

\end {proof}

 \begin{remark} When we are dealing with the double series $T(S(X))$ we actually should consider two different objects: the powers series formally obtained by manipulation of the coefficients, in line with the above theorems, and the composite function $T\circ S: x \to T(S(x)), \forall x$ in some domain.

The above Example \ref{non-substitution} shows that the composite function $T \circ S$ may be defined somewhere, while the formal double series does not exist because its coefficients are diverging series.

The following example shows that the formal power series, existing in this case,  may have a  domain which is different from the set $\{0\}$ but is, nevertheless, strictly smaller than  the domain of the composite function. This would be impossible if the internal series $S(X)$ had only positive coefficients, as shown by the proof of Theorem \ref{substitution}.

\end{remark}

Before we discuss the following example, we recall the so called Fa\`a di Bruno formula for the $n-th$ derivative of a composite function $f \circ g$ (\cite {Mennucci}):

$$D^{(n)}(f \circ g)(x) \sum_{\pi \in P_n} f^{|\pi|}((g(x))\Pi_{B \in \pi}g^{|B|}(x),$$
where 

1. $\pi$ varies in the set $P_n$ of all partitions of the set $\{1,\cdot \cdot \cdot ,n\}$,

2. $B$ varies in $\pi$,

3. $|A|$ denotes the cardinalty of the finite set $A$. 

Such a formula, that can proved by induction on $n$ (see\cite{Mennucci} ), can be easily seen to hold in the non-Archimedean case.

So the $n-th$ derivative is a sum of products of derivatives containing one and only one element $D^{(n)} f(g(x))(g'(x)^n)$, while all other elements have the form $(D^{(m)}f)( D^{(h_1)}g)(D^{(h_2)}g) \cdot \cdot \cdot (D^{(h_r)}g)$, where $m < n$ so that there is at least one $h_i > 1$.

\begin{example} Let $\KK$ be the usual ordered field $\widehat {\QQ(\e)}$ and put: $\omega = \e^{-1}$. Set: $$S(X) = \e-\omega X+\omega^2 X^2+\sum_{n = 3}^{\infty}\e^nX^n, T(X) = \sum_{n=0}^{\infty}X^n.$$
We easily see that $S(0) = \e, T(\e) = (1-\e)^{-1}-1-\e^2, S(1) = \e+\sum_{n = 3}^{\infty} \e^n$.

Therefore $T(S(1))$ (composite function $T \circ S$ computed at $1$) does exist. We want to show that $1$ does not belong to the domain of the power series $T(S(X))$ obtained by formal substitution. To this purpose we recall that $T \circ S$ and $T(S(X))$ coincide on a neighbourhood of $0$ (see \ref{substitution2}), so that the {\it expected} coefficients of $T(S(X))$ can be obtained by means of the Fa\`a di Bruno formula (see \cite  {CMV1}   ).  It  is therefore enough to show that the sequence $D^{n}(T \circ S)(0))$ does not converge to $0$ when $n$ tends to $\infty$. In fact it holds: $$S'(0) = -\omega, S"(0) = 2\omega^2, S^{(n)}(0) = n!\e^n, \forall n \ge 3,$$ while $T^{(n)}(S(1))$ differs from $n!$ by an infinitesimal element.

Now we use the Fa\`a di Bruno formula (\cite {Mennucci}) to see that $D^{n}(T o S)(0)$ is the sum of finitely many elements, among which $T^{(n)}(S(1))S'(0)^n$ is an infinite of the same order as $\omega^n$, hence larger that the order of all other elements (not exceeding the order of $\omega^{n-1}$). Therefore  $\lim_{n \to \infty}D^{n}(T o S)(0)) = \infty.$ 

We want to point out that the use of the Fa\`a di Bruno formula can, in this case,  easily be avoided by a direct computation which shows that in the $n^{th}$ derivative of $T \circ S$ the dominant term is of the form $T^{(n)}(S')^n$.

We wanto to point out that, since  the composite function $T \circ S$ is defined at $1$, it can be expanded into a power series around $1$ itself. But such a power series must have a domain with empty intersection with the domain of $T(S(X))$ ( if two power series expanding the same function are centered at different points and have intersecting domains, they have the same domain  (\cite{CMV1}). This differs from the archimedean situation.
\end{example}

\begin{remark}

As for a formal double series with domain larger than the domain of the composite function, we will see in the last section of this paper an example of a double series $T(S(X))$ with the following properties:

- $T(S(X))$ is a polynomial, so that its domain is the whole of $\KK$,

- $T\circ S$ is defined on a domain larger than the set $\{0\}$ but different from the whole of $\KK$.

\end{remark}

\begin{remark} \label{complex}

Corollary \ref{substitution2} has the following analog concerning power series on the field of complex numbers, i.e. analytic functions:

{\it{Consider the power series with complex coefficients $S(X)=\sum_ja_jX^j$ and $T(X)=\sum_ib_jX^i$. If their convergence radius are positive, and if $S(0)$ is strictly internal to the convergence circle of $T(X)$, then $T(S(z))=\sum_j d_jz^j$ (where $d_j$ are the expected coefficients, with converging series) , if $z$ belongs to a suitable neighborough of $0$.}}

The proof in the case $a_0 = 0$ is contained in \cite  {Cartan} Chap. I.2. In the case $a_0 \ne 0$ we can argue as follows.

Name $r_1$ and $r_2$ the convergence radii of $S(X)$ and $T(X)$, respectively. Since $S(0) =a_0$, by hypothesis we have $|a_0|<r_2$. Consider the two series 
$\overline S (X) = \sum_j {|a_j |X^j } $ and $\overline T (X) = \sum_i {|b_i |X^i }$, which are different from $S(X)$ and $T(X)$, but have the same convergence radii.
Since $\overline S(0) = |a_0|<R_2$, $\overline T(\overline S(z))$ is convergent at $z=0$, and also when $z$ satisfies to $\overline S(z) < r_2$. Since $\overline S(0) < r_2$ and $\overline S(z)$ is a continuous functions, this appens for $|z|<k$, for suitable $k$. Then the series $\overline T(\overline S(z))$ is convergent for $|z|<k$, and consequently 
$\overline T(\overline S(z))=\sum_i|b_i|(\sum_j|a_j|z|^j)^i$ is convergent. Then the conclusion follows by an application of Theorem \ref{substitution}, which is well known for the complex case (see \cite {Goursat}).
\end{remark}

\begin{remark}

Condition $a_0 = 0$ has the following consequence:   the expected coefficients $d_j=\sum_ib_ic_{ij} $, that a priori  are series, become  finite sums, and therefore the substitution of a series into another series can be performed disregarding convergence.

The above Theorem \ref{substitution2} ensures that, when $a_0 \ne 0$, convergence of $T(a_0)$ is necessary and sufficient to have convergent expected coefficients and so to present $T(S(X))$ as a formal power series, converging in a suitable neighbourhood of $0$.

We also want to observe that, if $d_0$ is a converging series, then all the $d_j$'s are converging series.

\end {remark}


\begin{thebibliography}{99}

 



\bibitem{Bourbaki}
N. Bourbaki
\emph{\' El\'ements de Math\'ematiques, Livre II, Alg\`ebre, chap. 6, Groupes et corps ordonn\' es},
Hermann, Paris  ~(1964) 

\bibitem{Cartan}
Henri Cartan
\emph{Theorie elementaire des fonctions analytiques d'une ou plusieurs variables complexes},
Hermann, Paris 1961


 
   
\bibitem{CMV1}
L. Corgnier, C. Massaza, P. Valabrega
\emph{ On the intermediate value theorem over a non-Archimedean field},
Le matematiche, Catania  ~(2013) , vol. LXIII , Fasc. II, pp. 227-248
  
\bibitem{Goursat}
E. Goursat,
\emph{A course in Mathematical Analysis- vol. I -translated by E. R. Hedrick},
 Dover Publications, Inc, New York ~(1904) 
  

   
\bibitem{Kelley}
  J.L.Kelley
   \emph{ General topology },
Springer Verlag,  Berlin Heidelberg New York Tokyo~(1955) 
 
\bibitem{Massaza}
C. Massaza
\emph{Sul completamento dei campi ordinati},
Rend. Sem. Mat. Univ. Polit. Torino \textbf{29}~(1969-70), 329--348

\bibitem{Massaza2}
C. Massaza
\emph{Serie di potenze a coefficienti in un campo non archimedeo},
Riv. Mat. Univ. Parma  \textbf{11}~(1970), 133--157

 \bibitem{Mennucci}
Mennucci,A. An intuitive presentation of Fa\`a di Bruno's formula, Scuola Normale Superiore, Pisa, April 5, 2011



   
    
\bibitem{Zeta}
El Jury ,Theory and Applications of the z-Transform Method , John Wiley  $\&$ Sons, NY, 1964




\end{thebibliography}
 \end{document}